\documentclass{amsart}
 \usepackage{amssymb,stmaryrd}
 \usepackage{tikz, tikz-cd, pgfplots, amsfonts}
 \usepackage[hidelinks]{hyperref}

 \newtheorem{theorem}{{\rm T\sc heorem}}[section]
 \newtheorem{lemma}[theorem]{{\rm L\sc emma}}
 \newtheorem{corollary}[theorem]{{\rm C\sc orollary}}
 \newtheorem{proposition}[theorem]{{\rm P\sc roposition}}

 \newtheorem{definition}[theorem]{{\rm D\sc efinition}}
 \newtheorem{example}[theorem]{\rm E\sc xample}
 
 \theoremstyle{definition}
 \newtheorem{remark}[theorem]{{\rm R\sc emark}}
  
  \newtheorem*{acks}{Acknowledgments}

\DeclareMathOperator{\aut}{Aut}
\DeclareMathOperator{\pic}{Pic}
\DeclareMathOperator{\bl}{Bl}

\DeclareMathOperator{\mw}{MW}
\DeclareMathOperator{\ns}{NS}
\DeclareMathOperator{\numer}{Num}

\begin{document}
 \title {Nodal rational curves on Enriques surfaces of base change type}
 \author[Simone Pesatori]{Simone Pesatori}
 \address{Università degli Studi Roma Tre\\
         Largo San Leonardo Murialdo, 1\\
         Rome\\
         Italy}

 \begin{abstract} Using lattice theory, Hulek and Sch\"utt proved that for every $m\in\mathbb{Z}_+$ there exists a nine-dimensional family $\mathcal{F}_m$ of K3 surfaces covering Enriques surfaces having an elliptic pencil with a rational bisection of arithmetic genus $m$. We present a purely geometrical lattice free construction of these surfaces, that allows us to prove that generically the mentioned bisections are nodal. Moreover, we show that, for every $m\in\mathbb{Z}_+$, the very general Enriques surface covered by a K3 surface in $\mathcal{F}_m$ admits a countable set of nodal rational curves of arithmetic genus $(4k^2-4k+1)m+4k^2-4k$ for every $k\in\mathbb{Z}_+$, that form a rank 8 subgroup of the automorphism group of the surface. As an application, we compute the linear class of the $n$-torsion multisection for every $n\in\mathbb{N}$ for a general rational elliptic surface. 

\end{abstract}
 \maketitle
 \pagestyle{myheadings}\markboth{\textsc{ Simone Pesatori }}{\textsc{Nodal rational curves on Enriques surfaces of base change type}}
\section{Introduction}

This work addresses the topic of nodal rational curves on Enriques surfaces of base change type. The interest of the results is two-fold: from one side it is a contribution to the study of rational curves on algebraic varieties; on the other hand it is linked to the study of Severi varieties of nodal curves on surfaces, where by nodal we mean that the curves have just simple nodes as singularities.\par
Concerning rational curves, it is well-known that the only curves with negative self-intersection that an Enriques surface can admit are the smooth $(-2)$-curves, which turn out to be rational. The Enriques surfaces admitting such curves are called \textit{nodal} (or \textit{special}) and they form a divisor in the ten-dimensional moduli space of Enriques surfaces, whose geometry has been studied by Dolgachev and Kondo in \cite{DK}, where the authors prove its rationality. It is classical that every Enriques surface is equipped with some genus 1 fibrations over a base isomorphic to $\mathbb{P}^1$, and that their singular fibers are generically genus 1 curves with a node, whence rational. Concerning higher genera, Hulek and Sch\"utt in \cite{HS} show that for every positive integer $m\in\mathbb{Z}_+$ there exist a nine-dimensional family $\mathcal{F}_m$ in the moduli space of K3 surfaces parametrizing K3's covering an Enriques surface admitting a rational curve of arithmetic genus $m$, called Enriques surfaces of base change type (or \textit{$m$-special} Enriques surfaces). More precisely, they prove that every Enriques surface covered by a K3 surface in $\mathcal{F}_m$ has an elliptic pencil with a rational bisection of arithmetic genus $m$ that splits in the K3 cover in two smooth $(-2)$-curves. Even regarding higher genera, Galati and Knutsen in \cite{GK} state a necessary condition for the existence of a rational curve in the very general Enriques surface: they prove that every rational curve in the very general Enriques surface $Y$ is 2-divisible in $\numer(Y)$. \par
The Severi varieties parametrize curves with a prescribed number of nodes in a fixed linear system. Severi varieties were introduced by Severi in \cite{Se}, where he proved that all Severi varieties of irreducible $\delta$-nodal curves of degree $d$ in $\mathbb{P}^2$ are nonempty and smooth of the expected dimension. Severi also claimed irreducibility of such varieties, but his proof contains a gap. The irreducibility was proved by Harris in \cite{H}. Severi varieties on other surfaces have received much attention in recent years, especially in connection with enumerative formulas computing their degrees (see for example \cite{B}, \cite{BOP} and \cite{CH}. Nonemptiness, smoothness, dimension and irreducibility for Severi varieties have been widely investigated on various rational surfaces (see, e.g.,
\cite{GLS}, \cite{Ta} and \cite{Te}, as well as K3, Enriques and abelian surfaces (see, e.g.,\cite{Che}, \cite{KL}, \cite{KLM}, \cite{LS}, \cite{BL}, \cite{MM}, \cite{CDGK} and \cite{CDGK2}). \par
In this work we link the two topics of rational curves on surfaces and of Severi varieties of curves on surfaces. First of all, we show that the rational bisections found by Hulek and Sch\"utt in \cite{HS} are generically nodal. In particular, we prove the following Theorem.
\begin{theorem}\label{nodal}
    Let $Y_m$ be a general $m$-special Enriques surface and let $B_{Y,m}\subset Y_m$ be an $m$-special curve. Then, $B_{Y,m}$ is nodal.
\end{theorem}
Furthermore, we see that fixed a general $m$-special Enriques surface $Y_m$, there are several nodal rational curves for an infinite range of arithmetic genera. The precise statement is described in the following Theorem.
\begin{theorem}\label{nodalk}
    For every $m\in\mathbb{Z}_+$, the very general $m$-special Enriques surface admits a nodal rational curve of arithmetic genus 
     \begin{center}
        $\rho_a(B_{E,k,m})=(4k^2-4k+1)m+4k^2-4k$
    \end{center}
 for every $k\in\mathbb{Z}_+$.
\end{theorem}
Lastly, we explain that fixed the genus 1 pencil $\epsilon:Y_m\rightarrow\mathbb{P}^1$ of which all these rational curves are bisections, for every $k\in\mathbb{Z}_+$, there is a one to one correspondence between the set of nodal rational bisection of arithmetic genus $(4k^2-4k+1)m+4k^2-4k$ and the group of sections $\mw(S)$ of a certain rational elliptic surface $S$ involved in the construction of $Y_m$. Since by a result due to Hulek and Sch\"utt, the group $\mw(S)$ injects into $\aut(Y_m)$, we conclude that the nodal rational bisections of a given arithmetic genus form a subgroup of $\aut(Y_m)$.
\par
In order to prove the existence of the Enriques surfaces of base change type, Hulek and Sch\"utt resort to the theory of lattices in a way that we partially explain in Section \ref{sec3}. To prove Theorem \ref{nodal} and Theorem \ref{nodalk}, we give a purely geometrical proof, that is, lattice free, of the existence of this particular kind of Enriques surfaces. More precisely, we reduce the problem to showing the nonemptiness of some particular Severi varieties of curves on some rational surfaces, for instance the rational elliptic surfaces and the second Hirzebruch surface $\mathbb{F}_2$.\par
The construction of Hulek and Sch\"utt starts from a rational elliptic surface, so that in Section \ref{sec2} we review the basics about rational elliptic surfaces, K3 surfaces and Enriques surfaces. In Section \ref{sec3} we recall and introduce some properties of the group law of an elliptic curve and of the group of sections of an elliptic surfaces: we give the notion of \textit{origin cutting} divisor, that will be crucial for the rest of the treatment. As an application, we compute the linear classes of all the torsion multisections for a general rational elliptic surface. In Section \ref{sec4} we introduce the main object of the work, the Enriques surfaces of base change type, and we explain why they are linked to some origin cutting linear systems on the rational elliptic surfaces involved in their construction. In Section \ref{sec5}, after recalling the basics about Severi varieties of curves on surfaces, we give our lattice free construction of the Enriques surfaces of base change type, thanks to which we are able to prove Theorem \ref{nodal} and Theorem \ref{nodalk}. The main tool will be the logarithmic Severi varieties of curves on surfaces, that roughly speaking parametrize curves with given tangency conditions to a fixed curve.

 \begin{acks} I want to thank the advisor of my PhD Thesis Andreas Leopold Knutsen for the fruitful conversations which inspired this work. I would also like to thank Concettina Galati for her help. \end{acks}

\section{Rational elliptic surfaces, K3 surfaces and Enriques surfaces \label{sec2}}

We recall the basics about elliptic surfaces, K3 surfaces and Enriques surfaces. As general references the reader might consult \cite{M}, \cite{BHPV}, \cite{CD} or \cite{HS}. \par

For us a \textit{genus one fibration} is a morphism $f:X\rightarrow C$, where $X$ is an algebraic surface and $C$ is a smooth curve, such that the general fiber is a smooth curve of genus one. If there is a section $s:C\rightarrow X$, we say that $f:X\rightarrow C$ is an \textit{elliptic fibration} (with a given section), and $X$ is an \textit{elliptic surface} over $C$.\par
If $X\rightarrow C$ is equipped with a chosen section $s$. Then the set of sections is an abelian group with the group addition defined fiberwise. The group of the sections of $X\rightarrow C$ is called Mordell-Weil group of the elliptic surface, denoted by $\mw(X\rightarrow C)$ or simply $\mw(X)$ if the surface has only one elliptic fibration or if it is clear to what fibration we are referring. The chosen section $s$, which is the zero element of $\mw(X)$, is called the zero-section.

\begin{example}
Let $C_1$ and $C_2$ be two smooth cubic curves in $\mathbb{P}^2$ and consider the pencil they generate. It has nine base points counted with multiplicity, corresponding to the nine intersection points between $C_1$ and $C_2$. Let $e:S=\tilde{\mathbb{P}}^2\rightarrow\mathbb{P}^2$ be the blow-up of $\mathbb{P}^2$ at the base points of the pencil of cubics. Then the pullback of the pencil is base point free and induces a morphism $f:S\rightarrow\mathbb{P}^1$. A general fiber of $f$ is the strict transform of a general
member of the pencil of cubics, which is a smooth elliptic curve. Then $f:S\rightarrow\mathbb{P}^1$ is a rational elliptic surface: the section $E$ can be chosen to be the exceptional divisor of the last blow-up of $S\rightarrow\mathbb{P}^2$. \end{example}
By \cite[Lemma IV.1.2]{M}, every rational elliptic surface arises in this way.\par

It is well-known that every rational elliptic surface has twelve singular fibers counted with multiplicity, and there is a wonderful theory of the configurations of singular fibers a rational elliptic surface can admits, developed by Miranda and Persson. The authors in  \cite{P} and \cite{M3} prove that there is a list of 279 possible confluences of singular fibers for a rational elliptic surface. In this work, unless differently specified, we will consider just rational elliptic surfaces with twelve nodal curves as singular fibers.
\begin{definition}
    We say that a rational elliptic surface is \textit{general} if it has twelve nodal curves as singular fibers.
\end{definition}

We call $E_1,\dots,E_9$ the exceptional divisors over the points $P_1,\dots,P_9$. With this notation, the Picard group of $S\cong\bl_{\{P_1,\dots,P_9\}}\mathbb{P}^2$ is 
\begin{center}$\pic(S)\cong\mathbb{Z}L\oplus\mathbb{Z}E_1\oplus\dots\oplus\mathbb{Z}E_9$,\end{center}
with $L$ the pullback of a line in $\mathbb{P}^2$.
If $S\cong \bl_{P_1,\dots,P_9}\mathbb{P}^2$ is a general rational elliptic surface, we choose the last exceptional divisor $E_9$ to be the zero-section of the fibration. 
 With the previous notations, the Mordell-Weil group of  $S$ is \begin{center} $\mw(S)\cong\mathbb{Z}^8$, \end{center}and it is generated by the exceptional divisors $E_1,\dots,E_8$. The neutral element is the zero-section $E_9$.

\begin{definition}
    A smooth projective surface $X$ is called \textit{K3 surface}
if $X$ is (algebraically) simply connected with trivial canonical bundle $\omega_X\cong\mathcal{O}_X$.\\
An \textit{Enriques surface} $Y$ is a quotient of a K3 surface $X$ by a fixed point free involution $\tau$. Such an involution is also called \textit{Enriques involution}.
\end{definition}

In a K3 surface $X$ (as well as in a rational elliptic surface) algebraic equivalence is the same as numerical equivalence and the Nerón-Severi group $\ns(X)$ is equipped with the structure of an even integral lattice of signature $(1, \rho(X)-1)$, where $\rho(X)$ is the Picard number of $X$.  
An Enriques surface $Y$ is not simply-connected; its fundamental group is $\pi_1(Y)=\mathbb{Z}/2\mathbb{Z}$.

Unlike in the K3 case, algebraic and numerical equivalence of divisors do not coincide on an Enriques surface $Y$: there is two-torsion in $\ns(Y)$ represented by the canonical divisor $K_Y$. The quotient $\ns(Y)_f$ of $\ns(Y)$ by its torsion subgroup is an even unimodular lattice, which is isomorphic to the so-called \textit{Enriques lattice}: \begin{center}
    $\numer(Y)=\ns(Y)_f\cong U\oplus E_8(-1)$.
\end{center}
Via pull-back under the universal covering, this lattice embeds primitively into $\ns(X)$. Here the intersection form is multiplied by two \begin{center}
    $U(2)\oplus E_8(-2)\hookrightarrow\ns(X)$.
\end{center}
Such K3 surfaces form a ten-dimensional irreducible moduli space.
\begin{definition}\label{verygenenr}
    We say that an Enriques surface $Y$ is \textit{Picard very general} if its universal covering $X$ is such that \begin{center}
        $\ns(X)\cong U(2)+E_8(-2)$
    \end{center}
\end{definition}
\begin{remark}
    An Enriques surface $Y$ is Picard very general if and only if the Picard rank of its universal covering $X$ is equal to $10$.
\end{remark}
K3 and Enriques surfaces are the only surfaces that may admit more than one genus 1 pencil.
It is well-known that every genus 1 pencil on an Enriques surface has exactly two fibers of multiplicity two, called \textit{half-fibers}. The canonical divisor of $Y$ can be represented as the difference of the supports of the half-fibers of a genus 1 pencil: if $F$ is a genus 1 pencil of $Y$ and \begin{center}
    $2E_1\sim F$ and $2E_2\sim F$, then $K_Y\sim E_1-E_2$.
\end{center}

The next result due to Galati and Knutsen (see \cite[Theorem 1.1]{GK}) states a necessary condition for the existence of a rational curve in the very general Enriques surface. Here very general just means that there exists a set that is the complement of a countable union of proper Zariski-closed subsets in the moduli space of Enriques surfaces satisfying the given conditions. 
\begin{theorem}\label{galknu}(Galati, Knutsen)
    Let $Y$ be a very general Enriques surface. If $C\subset Y$ is an irreducible rational curve, then $C$ is 2-divisible in $\numer(Y)$.
\end{theorem}
As a consequence of this Theorem, if $C$ is a rational curve on the very general Enriques surface $Y$, then there are no curves on $Y$ intersecting $C$ in just one point.

\section{Origin cutting divisors on elliptic surfaces \label{sec3}}

We briefly recall the notions and the main properties of the elliptic curves and some properties of the elliptic surfaces we did not point out in Section \ref{sec2}. As general references the reader might consult \cite{Si}, \cite{M} and \cite{SS}

It is a very well-known fact that every elliptic curve has a group law, that we denote by $\boxplus$, with origin a chosen point $O$. Every smooth elliptic curve can be embedded in the projective plane as a smooth cubic curve: the next classical result describes how the group law works in this context.\par Let $(E,O)$ be a smooth plane cubic: we call $P_O$ the third intersection point between $E$ and the tangent line to $E$ at $O$.
\begin{proposition}[Group law for cubic curves]\label{lawcubic}
  Let $D_1$ and $D_2$ be divisors in $\pic(\mathbb{P}^2)$ such that $D_1\sim D_2$, that \begin{center} $D_{1_{|E}}=a_1Q_1+\dots+a_kQ_k$ \end{center} and that \begin{center}$D_{2_{|E}}=b_1R_1+\dots+b_mR_m$ \end{center} with the $a_i$'s and the $b_j$'s integers. Then \begin{center}
      $a_1Q_1\boxplus\dots\boxplus a_kQ_k=b_1R_1\boxplus\dots\boxplus b_mR_m$.
  \end{center}  In particular, if $D$ is a degree $d$ curve in $\mathbb{P}^2$ intersecting $E$ in $3d$ (not necessarily distinct) points $Q_1,\dots,Q_{3d}$, then\begin{center}
      $Q_1\boxplus\dots\boxplus Q_{3d}=dP_O$.
  \end{center}  
\end{proposition}
If the origin $O$ is an inflection point for the cubic $E$, then the sum of the intersection points between $E$ and any curve is $O$.\begin{definition}
 An \textit{n-torsion point} of $(E,O)$ is a point $Q\in E$ such that $Q^{\boxplus n}=O$.
\end{definition}
\begin{lemma}
    Every smooth elliptic curve has $n^2-1$ nontrivial $n$-torsion points.
\end{lemma}
Every elliptic curve admits a natural involution $(-1)$, that acts by interchanging opposite points with respect to the origin $O$. The fixed locus of $(-1)$ is composed by $O$ and the three nontrivial 2-torsion points.\par
Let now $X\rightarrow C$ be an elliptic surface with a given section $s_0$ that we choose to be the zero-section of the fibration. For every $t\in C$, we call $0_t$ the intersection point between the fiber $X_t$ and the section $s_0$. We give the definition of torsion multisection for $X\rightarrow C$.
\begin{definition}[Torsion multisection]
    Let $\tilde{\mathcal{H}}_m$ be the closure of the locus in $X$ of points $P_t\in X_t$, with $X_t$ smooth elliptic fiber, such that $P_t^{\boxplus m}=0_t$.\\ We define $\mathcal{H}_m:=\tilde{\mathcal{H}}_m-s_0$ to be the $m$-torsion multisection of $X$.
\end{definition}
It is clear that $\mathcal{H}_m$ is an $(m^2-1)$-section for $X$.
\begin{remark}\label{hke9}
    We point out that $\mathcal{H}_m$ does not intersect the zero-section $s_0$. The proof is essentially performed by Miranda in \cite[Proposition VII.3.2]{M}: the author proves that if a torsion section meets the zero-section in an elliptic fibration, then the two section coincides; the identical argument shows that if a torsion multisection meets the zero-section, then the multisection has the zero-section as irreducible component.
\end{remark}

The choice of the last exceptional divisor $E_9$ as zero-section for a general rational elliptic surface $S$ in particular means that for every $t$ in the base $\mathbb{P}^1$, the origin of the fiber $F_t$ is its intersection with $E_9$. Let us call $O_t$ this point and $P_t$ the third intersection point between $F_t$ and the tangent line to $F_t$ at $O_t$. \par

We introduce the notion of origin cutting linear systems for a general rational elliptic surface $S$.
\begin{definition}
    Let $\mathcal{L}\in\pic(S)$ be a divisor such that \begin{center}
        $\mathcal{L}_{|F_{t}}=\sum\limits_{i=1}^r a_{i,t}Q_{i,t}$ \end{center} with $a_{i,t}\in\mathbb{Z}$ for every $i$ and for every $t\in\mathbb{P}^1$. We say that $\mathcal{L}$ is \textit{origin cutting} if \begin{center} $Q_{1,t}^{\boxplus a_{1,t}}\boxplus\dots\boxplus Q_{r,t}^{\boxplus a_{r,t}}=O_{t}$ \end{center}
        for every $t\in\mathbb{P}^1$.\\
    If $\mathcal{L}$ is effective and such that $\mathcal{L}\cdot F_t=k$, we sometimes refer to it as an origin cutting $k$-section.
\end{definition}
\begin{remark}
    By Proposition \ref{lawcubic}, the property of a divisor being origin cutting only depends on its linear class. For this reason, we can extend the notion of origin cutting divisor to the linear systems. 
In other words, the origin cutting linear systems consist of divisors cutting each curve of the elliptic pencil in points whose sum in the group law is the origin. \end{remark}
\begin{lemma}\label{tors}
    The torsion multisections are origin cutting.
\end{lemma}
\begin{proof}
    The sum of the $m^2-1$ nontrivial $m$-torsion points of an elliptic curve is the origin of the group law.
\end{proof}
The next Lemma states one of the main properties of these systems. 
\begin{lemma}\label{mpl}
    If an origin cutting $k$-section $B$ of a general rational elliptic surface $S\cong \bl_{P_1,\dots,P_9}\mathbb{P}^2$ has a $k$-ple point $Q$, then either $Q\in\mathcal{H}_k$ or $Q\in E_9$. 
\end{lemma}

\begin{proof}
    If $Q\in B$ is $k$-ple, then $Q^{\boxplus k}=0$ by definition of origin cutting divisor.
\end{proof}

\begin{remark}\label{even}
It is obvious that an origin cutting $k$-section $B$ cannot have a $(k+1)$-ple point. In particular, this implies that an origin cutting bisection can just have double points as singularities. 
\end{remark}

The next Proposition describes more precisely the origin cutting linear systems.
\begin{proposition}\label{forminv}
    Let $S\cong\bl_{\{P_1,\dots,P_9\}}\mathbb{P}^2$ be a general rational elliptic surface and let $\mathcal{L}\in\pic(S)$ be an effective divisor. Then, $\mathcal{L}$ is an origin cutting $k$-section (without $E_9$ as irreducible component) if and only if it is of the form \begin{center} $\mathcal{L}\sim 3(c+k)L-(c+k)E_1-\dots-(c+k)E_8-cE_9$,\end{center} for some $c\in\mathbb{Z}_+$.
\end{proposition}
\begin{proof}
    $\mathcal{L}\in\pic(S)$, then \begin{center}
        $\mathcal{L}\sim aL-b_1E_1-\dots-b_8E_8-cE_9$.
    \end{center}
    Since $\mathcal{L}$ is origin cutting, we have that for every $t\in\mathbb{P}^1$,\begin{center}
        $\mathcal{L}_{|F_{t}}=Q_{1,t}+\dots+Q_{k,t}$
    \end{center}
    in such a way that \begin{center}
        $Q_{1,t}\boxplus\dots\boxplus Q_{k,t}=O_{t}$.\end{center}
If we call $\overline{\mathcal{L}}$ and $\overline{F}$ the pushforwards of $\mathcal{L}$ and $F$ under $S\rightarrow\mathbb{P}^2$, we have that     
 \begin{center}
            $\overline{\mathcal{L}}_{|\overline{F}_{t}}=b_1P_1+\dots+ b_8P_8+ cP_9+ Q_{1,t}+\dots+ Q_{k,t}$.
        \end{center}
 But $\overline{\mathcal{L}}$ has degree $a$, whence \begin{center}
            $b_1P_1\boxplus\dots\boxplus b_8P_8\boxplus cP_9\boxplus Q_{1,t}\boxplus\dots\boxplus Q_{k,t}=b_1P_1\boxplus\dots\boxplus b_8P_8=aP_{t}$ for every $t$.
        \end{center}
        Furthermore, by Proposition \ref{lawcubic}, for every $t$ we have,  \begin{center}
            $P_1\boxplus\dots\boxplus P_8=3P_{t}$:
        \end{center}indeed, $P_1,\dots,P_9$ are the base points of the pencil of cubics and by the choice of $E_9$ as zero-section, we have that the origin of $\overline{F}_t$ is $P_9$.
        By the generality of $S$ (and then of $P_1,\dots,P_9$ as base points of a pencil of cubics), we deduce
        \begin{center}
            $b_1=\dots=b_8=:b$.
        \end{center} Putting the conditions together, we get $a=3b$. Lastly, since $\mathcal{L}$ is a $k$-section, we have that $3a-8b-c=b-c=k$, so that $b=c+k$ and this completes the proof.
    
\end{proof}
As an application, we compute the linear class of the torsion multisection $\mathcal{H}_n$ for a general rational elliptic surface $S\cong\bl_{\{P_1,\dots,P_9\}}\mathbb{P}^2$ for every $n\in\mathbb{Z}_+$.
\begin{proposition}\label{linearclassmult}
    For every $n\in\mathbb{Z}_+$, the $n$-torsion multisection is an $(n^2-1)$-section of the form \begin{center}
    $H_n\sim 3(n^2-1)L-(n^2-1)E_1-...-(n^2-1)E_8$.
\end{center} 
\end{proposition}
\begin{proof}
  Lemma \ref{tors} ensures that the torsion multisections are origin cutting. Then, for every $n\in\mathbb{Z}_+$, we have that $\mathcal{H}_n$ is of the form \begin{center}
    $H_n\sim 3(n^2+c-1)L-(n^2+c-1)E_1-...-(n^2+c-1)E_8-cE_9$
\end{center} for some $c\in\mathbb{Z}_+$. Moreover, by Remark \ref{hke9} we have $\mathcal{H}_n\cdot E_9=0$, so that $c=0$.
\end{proof}

\begin{remark}\label{fixlocus}
In particular, Proposition \ref{linearclassmult} implies that the 2-torsion trisection $\mathcal{H}_2$ is \begin{center}
    $\mathcal{H}_2\sim 9L-3E_1-\dots-3E_8$, \end{center}
as showed by using other methods by Vakil in \cite{V}.
\end{remark}

We now discuss an example relevant to the rest of the work.

\begin{example}[Origin cutting bisections]\label{invbis}
For $k=2$, the form of an origin cutting bisection $\mathcal{B}_m$ such that $\mathcal{B}_m\cdot E_9=2m$ for some $m\in\mathbb{Z}_+$ is \begin{center}
    $\mathcal{B}_m\sim 6(m+1)L-2(m+1)E_1-\dots-2(m+1)E_8-2mE_9$.
\end{center}
We will be interested in rational members of these bisections: by Proposition \ref{mpl}, the double points of the bisections belong either to $\mathcal{H}_2$ or to $E_9$. 
\end{example}

 Every rational elliptic surface $S\cong \bl_{P_1,\dots,P_9}\mathbb{P}^2$ carries a natural involution $(-1)\in\aut(S)$, that acts fiberwise by interchanging opposite points with respect to the group law with origin $O_t$. This involution is known in the literature as the Bertini involution. As pointed out at the beginning of this section, for every $t\in\mathbb{P}^1$ such that $F_t$ is smooth, the points fixed by $(-1)$ are $O_t$ and the three nontrivial 2-torsion points of $F_t$. The fixed locus of $(-1)$ is the union of the zero-section $E_9$ and the trisection $\mathcal{H}_2$ parametrizing the nontrivial 2-torsion points of any smooth fiber. We describe the quotient of a rational elliptic surface by the involution $(-1)$. We call $q$ the quotient map $q:S\rightarrow S/(-1)$. This result is classical: see for example \cite[Proof of Proposition 3.2]{V}, \cite[Section 4.4, p.408]{CD} or \cite[Section 2]{DOPW}.

\begin{proposition}\label{quot}
    The quotient $S/(-1)$ is isomorphic to the second Hirzebruch surface $\mathbb{F}_2\cong\mathbb{P}(\mathcal{O}_{\mathbb{P}^1}\oplus\mathcal{O}_{\mathbb{P}^1}(-2))$.
\end{proposition}

We conclude the section by showing that the image of an origin cutting bisection under the involution $(-1)$ is the bisection itself.
 
\begin{lemma}
    The origin cutting bisections are invariant with respect to the involution $(-1)$.
\end{lemma}
\begin{proof}
    $\mathcal{B}_m$ cuts every fiber $F_{t}$ in two points $Q_{1,t}$ and $Q_{2,t}$, in such a way $Q_{1,t}\boxplus Q_{2,t}=O_{t}$, or, equivalently, $Q_{1,t}=\boxminus Q_{2,t}$. This implies that for every origin cutting bisection $B\in|\mathcal{B}_m|$, we have $(-1)^*(B)=B$.
\end{proof}

\section{Enriques surfaces of base change type \label{sec4}}
We present the the main features of Enriques surfaces of base change type. We firstly describe their construction due to Hulek and Sch\"utt, then we will focus on the rational bisections they have and lastly we show their connection with some origin cutting bisections of the rational elliptic surface involved in their construction. Part of the proofs have been performed by Hulek and Sch\"utt in \cite{HS}: in order to make the reader familiar with the geometric ideas and the notations, we report and sometimes extend them by adding some details that will be important for the continuation.\par
Let $S=\bl_{\{P_1,\dots,P_9\}}\mathbb{P}^2$ denote a rational elliptic surface. We let now \begin{center}
    $g:\mathbb{P}^1\rightarrow\mathbb{P}^1$
\end{center} be a morphism of degree two. Denote the ramification points by $t_0$ and $t_{\infty}$. It is well-known that the pullback of $X$ via $g$ is a K3 surface under the assumption that all the irreducible components of the fibers of $S$ over $t_0$ and $t_{\infty}$ are reduced. With abuse of notation, we denote by $g$ also the double cover $X\rightarrow S$ and we denote by $\tilde{E}_i$ the pull-backs of the exceptional divisors $E_i$ of $S$. With this notation, $\tilde{E}_9$ is the zero-section for the induced elliptic fibration on $X$.\par Moreover, we denote by $S_t$ the fiber on $S$ over a point $t\in\mathbb{P}^1$ and by $X_t$ and $X_{-t}$ the two components of its preimage on $X$. Since $S_t\cong X_t\cong X_{-t}$, if $Q_t\in S_t$, we denote the two points in its preimage $g^{-1}(Q_t)$ by $\tilde{Q_t}$ and $\tilde{Q}_{-t}$. Sometimes, we refer to the pair $X_t$ and $X_{-t}$ as \textit{twin fibers}, to the pair $\tilde{Q_t}$ and $\tilde{Q}_{-t}$ as \textit{twin points in twin fibers} and to the pair $\tilde{Q_t}$ and $\boxminus\tilde{Q}_{-t}$ as \textit{opposite points in twin fibers} (with respect to $\tilde{E}_9$). \par
Let $\iota$ denote the deck transformation for $g$, i.e. $\iota\in\aut(\mathbb{P}^1)$ such that $g=g\circ\iota$. 
Then $\iota$ induces an automorphism of $X$ that we shall also denote by $\iota$. The quotient $X/\iota$ returns exactly the rational elliptic surface $S$ we started with. 

\begin{remark}
    We obtain a ten-dimensional family of elliptic K3 surfaces: eight dimensions from the rational elliptic surfaces and two dimensions from the base change.
\end{remark}

\begin{definition}\label{vergenbaschan}
    We say that such a base change $g:X\rightarrow S$ is \textit{very general} if $S$ is general, $S_{t_0}$ and $S_{t_{\infty}}$ are smooth elliptic curves and $\iota^*$ acts as the identity on $\ns(X)$.\par In this case we also say that $X$ is 
\textit{base change very general} as K3 surface of base change type.
\end{definition}

\begin{proposition}
    A base change very general K3 surface $X$ does not carry any Enriques involution.
\end{proposition}

\begin{proof}
    $\ns(X)\cong U\oplus E_8(-2)$ (see \cite[Section 3.2]{HS}). Since the general Enriques lattice does not embed primitively into $U\oplus E_8(-2)$, $X$ cannot admit an Enriques involution. 
\end{proof}

Hulek and Sch\"utt in \cite{HS} impose a geometric condition on the base change $g:X\rightarrow S$ that allows to construct (a countable number of) families of K3 surfaces with an Enriques involution. In order to exhibit K3 surfaces with Enriques involution within our family of K3 surfaces of base change type, we need the following Lemma, that summarizes the discussion in \cite[Section 3.3]{HS}.
\begin{lemma}[Hulek-Sch\"utt]
    Let $P$ be a section for the elliptic fibration on $X$ induced by the one of $S$. Then \begin{itemize}
        \item either $P$ is invariant with respect to $\iota^*$,
        \item or $P$ is anti-invariant with respect to $\iota^*$ (meaning that $\iota^*(P)=(-1)^*(P)$,where $(-1)$ indicates the involution on $X$ acting fiberwise by interchanging opposite points with respect to the zero-section $\tilde{E}_9$).
    \end{itemize}
    Moreover, in the former case, $P$ is the pull-back of a section $E\in\mw(S)$ and it cuts twin points in twin fibers, while in the latter $P$ cuts opposite points in twin fibers.
\end{lemma}

 We denote by $\boxplus P\in\aut(X)$ the involution on $X$ given by the fiberwise translation for $P$.

\begin{proposition}\label{p -p}
   Let $P$ be an anti-invariant section with respect to $i^*$. Then, the quotient $g:X\rightarrow S$ identifies $P$ with its opposite section $\boxminus P$ with respect to $\tilde{E}_9$, as well as $P^{\boxplus k}:=P\boxplus\dots\boxplus P$ with $P^{\boxminus k}:=\boxminus P\boxminus\dots\boxminus P$.  
\end{proposition}
\begin{proof}
    For every $t\in\mathbb{P}^1$, if $P$ cuts a fiber $X_t$ in a point $\tilde{Q}_t$, the opposite section $\boxminus P$ cuts $X_t$ in the opposite point $\boxminus\tilde{Q}_t$. 
     Moreover, $P$ intersects the twin fiber $X_{-t}$ in $\boxminus\tilde{Q}_{-t}$, while $\boxminus P$ cuts $X_{-t}$ in $\tilde{Q}_{-t}$. To complete the proof, it is sufficient to notice that $g$ identifies $\tilde{Q}_t$ and $\tilde{Q}_{-t}$ as well as $\boxminus\tilde{Q}_t$ and $\boxminus\tilde{Q}_{-t}$. In the same way one can prove that $P^{\boxplus k}$ is identified with $P^{\boxminus k}$.
\end{proof}

\begin{remark}
    If there exists $P\in\mw(X\rightarrow\mathbb{P}^1)$ that is anti-invariant with respect to $\iota^*$, then $X\rightarrow S$ is not very general in the sense of Definition \ref{vergenbaschan}. Indeed, by Proposition \ref{p -p} we have $\iota^*(P)=\boxminus P\nsim P$, whence $\iota^*$ does not act as the identity on $\ns(X)$.
\end{remark}

 Consider the following automorphism of $X$ \begin{center}
    $\tau:=\iota\circ (\boxminus P)$.
\end{center}
\begin{proposition}[Hulek-Sch\"utt]\label{enrinv}
    $\tau\in\aut(X)$ is an involution and it is an Enriques involution if and only if $P$ does not intersect $\tilde{E}_9$ along $X_{t_0}$ and $X_{t_{\infty}}$.
\end{proposition}

\begin{definition}In the above set-up, let $Y=X/\tau$ denote the associated Enriques surface. We say that $Y$ is an \textit{Enriques surfaces of base change type} and we denote by $f$ the quotient $X\rightarrow Y$. \end{definition}

\begin{remark}The given elliptic fibration on $X$ induces a genus 1 fibration on $Y$. Here the smooth fiber $Y_t$ of $Y$ at $t$ is isomorphic to the fibers $X_t$ and $X_{-t}$ at $g^{-1}(t)$ as genus 1 curves or to the fiber of the rational elliptic surface $S_t$.\end{remark}

\begin{lemma}\label{pide9}(Hulek, Sch\"utt)
    The sections $\tilde{E}_9$ and $P$ of the specified elliptic fibration on $X$ are identified under $\tau$ and thus give a rational bisection for the induced genus 1 fibration on $Y$.
\end{lemma}
\begin{proof}
    Let $X_t$ and $X_{-t}$ be two twin fibers. The involution $\tau$ acts on a point $x_t\in X_t$ in the following way:\begin{center}
        $\tau(x_t)=\iota\circ\boxminus P(x_t)=\iota(x_t\boxminus P_t)=x_{-t}\boxplus P_{-t}$.
    \end{center}A point $0_t\in\tilde{E}_9$ is sent to $0_{-t}\boxplus P_{-t}=P_{-t}$ and a point $P_t\in P$ is sent to $P_{-t}\boxminus P_{-t}=0_{-t}$, where in this latter case the sign is changed because $P$ is $\iota^*$ anti-invariant, whence $\iota(P_t)=\boxminus P_{-t}$. Finally, since $P$ and $\tilde{E}_9$ are rational, then their image is. 
\end{proof}

The construction strongly depends on the choice of the $\iota^*$ anti-invariant section $P$. One could ask if such a section actually exists. The proof of their existence is performed by Hulek and Sch\"utt in a lattice-theoretical way (see \cite[Section 3]{HS}). In particular, the following Theorem collects their results in this context.
\begin{theorem}[Hulek, Sch\"utt]\label{sigmam}
    For every nonnegative integer $m\in\mathbb{Z}_+$, there exists a 9-dimensional family $\mathcal{F}_m$ of K3 surfaces such that, for every $X_m\in\mathcal{F}_m$, \begin{center}
    $\ns(X_m)\cong U\oplus E_8(-1)\oplus <-4(m+1)>$. 
\end{center}Moreover, $X_m$ covers an Enriques surface of base change type such that, with the previous notations, $P\cdot\tilde{E}_9=2m$.
\end{theorem}
\begin{remark}\label{P odd}
    The group of sections $\mw(X_m)$ is generated by the pullbacks $\tilde{E}_i$ and $P$. One could choose $P^{\boxplus k}$ (with $k$ odd for reasons that we explain in Remark \ref{remkodd}) to define the Enriques involution. To avoid confusion in the notations, we will choose the section $P$ to be not divisible in $\mw(X)$. 
\end{remark}

\begin{corollary}
    The Enriques surfaces of base change type are not Picard very general.
\end{corollary}
\begin{proof}
    By Theorem \ref{sigmam}, the Neròn-Severi group of $X_m\in\Sigma_m$ is \begin{center}
    $\ns(X_m)\cong U\oplus E_8(-1)\oplus <-4(m+1)>$, 
\end{center} and in particular $\rho(X_m)=11$. 
\end{proof}

\begin{definition}We denote by $B_{Y,m}$ the induced rational bisection on $Y_m$ and we say that $B_{Y,m}$ is an $m$-special curve for $Y_m$. Sometimes, we shall say that $Y_m$ is an $m$-special Enriques surface and that the induced genus 1 pencil on $Y_m$ is an $m$-special genus 1 pencil. \end{definition}

Since $B_{Y,m}^2=\frac{1}{2}(P+\tilde{E}_9)^2=\frac{1}{2}(4m-4)=2m-2$, the arithmetic genus of $B_{Y,m}$ is $\rho_a(B_{Y,m})=\frac{1}{2}B_{Y,m}^2-1=m$.

The next Theorem links the rational bisection $B_{Y,m}$ to the rational elliptic surface $S$ we started with to construct $Y_m$. More precisely, we are going to prove that the image of $P$ on $S$ under the quotient $g$ is an origin cutting bisection.

\begin{theorem}\label{bybs}
    Let $S\cong\bl_{\{P_1,\dots,P_9\}}\mathbb{P}^2$ be a general rational elliptic surface, and let $X_m$ and $Y_m$ be a K3 surface and an Enriques surface obtained by the base change construction. Let then $g:X_m\rightarrow S$ and $f:X_m\rightarrow Y_m$ denote the corresponding quotients. Finally, let $B_{Y,m}$ be the $m$-special curve of $Y_m$ and let $P$ and $\tilde{E}_9$ be the two components of its preimage under $f$. Then, $B_{S,m}:=g(P)$ is a rational bisection for the elliptic pencil on $S$. Moreover, $B_{S,m}\sim\mathcal{B}_m:= 6(m+1)L-2(m+1)E_1-\dots-2(m+1)E_8-2mE_9$ is an origin cutting bisection. 
\end{theorem}
\begin{proof}
    The section $P$ is anti-invariant with respect to $\iota^*$, then it cuts opposite points $\tilde{Q}_t$ and $\boxminus \tilde{Q}_{-t}$ in twin fibers $X_t$ and $X_{-t}$ with respect to $\tilde{E}_9$. Hence, $B_{S,m}$ cuts the fiber $S_t=g(X_t)=g(X_{-t})$ in the opposite points $Q_t$ and $\boxminus Q_t$ with respect to $E_9$ and then it is an origin cutting bisection. More over, since $P$ is rational, then $B_{S,m}$ is.
\end{proof}

In the following Remark, we give a geometrical interpretation of the phenomenon.
\begin{remark}\label{remkodd}
 By Proposition \ref{p -p}, the bisection $B_{S,m}\subset S$ splits in $P$ and $\boxminus P$ in $X_m$. This means that $B_{S,m}$ is tangent to the branch locus of $g$, that is the union of $S_{t_0}$ and $S_{t_{\infty}}$. 
Geometrically, since $B_{S,m}$ is a bisection for the elliptic fibration on $S$, it carries a $2:1$ map over $\mathbb{P}^1$. By the Riemann-Hurwitz formula, it has two ramification points, that correspond exactly to the two fixed fibers $S_{t_0}$ and $S_{t_{\infty}}$, to which $B_{S,m}$ is tangent. \par
Since $B_{S,m}$ is origin cutting, it has to be tangent to $S_{t_0}$ and $S_{t_{\infty}}$ along the 2-torsion trisection $\mathcal{H}_2$ or along $E_9$. It is easy to see that it is tangent to the fixed locus along $\mathcal{H}_2$: indeed, by Proposition \ref{enrinv}, to produce an Enriques involution $\tau$, $P$ cannot intersect $E_9$ along the fixed locus. For this reason in Remark \ref{P odd} we claimed that one could choose also $P^{\boxplus k}$ as section anti-invariant with respect to $\iota^*$ to define an Enriques involution, but with $k$ odd: if $k$ is even, then $P^{\boxplus k}$ intersects $E_9$ along the fixed locus.
\end{remark}

\begin{corollary}\label{freeh2}
The set of irreducible rational curves in $|\mathcal{B}_m|$ intersecting $E_9$ in simple nodes and tacnodes is nonempty.\end{corollary}
\begin{proof}
    $P$ and $\boxminus P$ cut $\tilde{E}_9$ in the same points: in fact, $P$ and $\boxminus P$ cut a fiber $X_t$ in opposite points, and if $P$ meets $X_t$ along $\tilde{E}_9$, the intersection point is the origin of the group law of the fiber and whence it coincides with its opposite point. This corresponds to a double point of $B_{S,m}$ along $E_9$. Since $P\cdot E_9=2m$, we have that $B_{S,m}\cdot E_9=2m$ and that this latter intersection is composed by (possible not ordinary) double points. The singularities of $B_{S,m}$ along $E_9$ can just be (possibly not ordinary) nodes but not cusps. Otherwise, it could not split. This means that the possible singularities are simple nodes and tacnodes.
\end{proof}
Notice that the presence of a tacnode would imply that $P$ (as well as $\boxminus P)$ and $\tilde{E}_9$ are tangent in $X_m$ and then that also the $m$-special curve $B_{Y,m}$ would have a tacnode.
One of the aims of the next section is to prove that generically this not happens, meaning that generically the $m$-special curves are nodal.

\begin{remark}[Description of the families]\label{descr}

An Enriques surface is called nodal if it contains a nodal curve, i.e. a curve of self-intersection $-2$, that turns out to be rational and smooth. On the K3 cover, such a curve splits into two disjoint smooth rational curves, again with self-intersection $-2$. For the case $P\cap E_9=\emptyset$, whence $m=0$, the construction due to Hulek and Sch\"utt thus leads exactly to special Enriques surfaces. \par
$\Sigma_1$ parametrizes K3 surfaces covering Enriques surfaces with a $1$-special curve $B_Y$, that is a rational curve of arithmetic genus 1. In other words, $B_Y^2=0$ and then either $|B_Y|$ or $|2B_Y|$ is a genus 1 pencil on $Y_1$, with $B_Y$ as one of its singular fibers. As usual $f:X_1\rightarrow Y_1$ indicates the Enriques quotient. Since $B_Y$ splits in $X_1$ in two $(-2)$-curves meeting at two points, it is 2:1 covered by a member of $|f^*(B_Y)|$ and thus it is an half-fiber of the genus 1 pencil $|2B_Y|$. The $Y_1$'s are precisely the Enriques surfaces having a genus 1 pencil with a nodal half-fiber: it is not surprising that they live in a subfamily of codimension 1 in the moduli space of the Enriques surfaces.\par

By Theorem \ref{galknu}, every rational curve in a very general Enriques surface is 2-divisible. For $m\geq 2$, we have that $X_m$ parametrizes K3 surfaces covering Enriques surfaces $Y_m$ having at least a linear system $L$ with $\rho_a(L)=m$ and an elliptic pencil $F$ with $L\cdot F=2$, with a rational member. In particular, if $E$ is one of the two half-fibers of $F$ (i.e. $2E\sim F$), we have $L\cdot E=1$. Thus, the peculiarity of $Y_m$ is the existence of a not $2$-divisible linear system of arithmetic genus $m$ having a rational member.
\end{remark}

\section{Logarithmic Severi varieties \label{sec5}}

In this section we recall the basics about Severi varieties of curves on surfaces, particularly focusing on the case of Enriques surfaces. Later, we give the notion of logarithmic Severi variety of curves on a fixed surface. Once clarified the connection between the existence of the Enriques surfaces of base change type and the nonemptiness of some Severi varieties of curves on rational elliptic surfaces and the second Hirzebruch surface $\mathbb{F}_2$, we prove Theorem \ref{nodal} and \ref{nodalk}.
Let $S$ be a smooth complex projective surface and $L$ a line bundle on $S$ such that the complete linear system $|L|$ contains smooth, irreducible curves (such a line bundle,
or linear system, is often called a Bertini system). Let
\begin{center}
    $\rho:=\rho_a(L)=\frac{1}{2}L\cdot(L+K_S)+1$
\end{center} be the arithmetic genus of any curve in $|L|$.

\begin{definition}\label{severi}
    For any integer $0\leq\delta\leq\rho$, consider the locally closed, functorially defined subscheme of $|L|$ \begin{center}
        $V_{|L|,\delta}(S)$ or simply $V_{|L|,\delta}$
    \end{center}parametrizing irreducible curves in $|L|$ having only $\delta$ nodes as singularities: this is called the \textit{Severi variety} of $\delta$-nodal curves in $|L|$. We will let $g:=\rho-\delta$ be the geometric genus of the curves in $V_{|L|,\delta}$.
\end{definition}

 It is well-known that, if $V_{|L|,\delta}$ is nonempty, then all of its irreducible components $V$ have dimension $\dim(V)\geq\dim|L|-\delta$. If $V_{|L|,\delta}$ is smooth of dimension $\dim|L|-\delta$ at $[C]$ it is said to be \textit{regular at $[C]$}. An irreducible component $V$ of $V_{|L|,\delta}$ will be said to be \textit{regular} if the condition of regularity is satisfied at any of its points, equivalently, if it is smooth of dimension $\dim|L|-\delta$.\par
In this section, we contribute in the study of the Severi varieties of curves on Enriques surfaces and rational elliptic surfaces. We give the state of the art about the Severi varieties of curves on these surfaces. About the Enriques case, we refer to \cite{CDGK} and \cite{CDGK2}. 
The next result states the main property of the Severi varieties on Enriques surfaces.\par
Let $Y$ be an Enriques surfaces, $X$ be its K3 cover, $f:X\rightarrow Y$ denote the quotient map and $\tau$ denote the Enriques involution. 
Let now $V$ be an irreducible component of a Severi varieties of $\delta$-nodal curves on $Y$. Ciliberto, Dedieu, Galati and Knutsen in \cite[Proposition 1]{CDGK} prove that if $V$ is regular, then the curves in $V$ are covered by irreducible curves of $X$, while if $V$ is nonregular, then each curve $C$ of $V$ splits in $X$ in two curves $C'$ and $C''$. Moreover, they show that if $Y$ is very general in moduli, then $C'$ and $C''$ are linearly equivalent. 

In particular, a nodal rational curve in the very general Enriques surface split in two linearly equivalent rational curves on the K3 cover.\par
Recall that any rational elliptic surface is isomorphic to the blow-up of $\mathbb{P}^2$ in nine points that are base points of a pencil of cubics. There are some general results about the Severi varieties on curves in blown-up planes (for example \cite{GLS}), but in the setting in which the blown-up points are in general position. This does not cover our case, in which $P_1,\dots,P_9$ are base points of a pencil of cubic curves.

\subsection{\rm T\sc he m-special curves}

As noticed at the end of the previous section, proving that $B_{Y,m}$ is nodal, consequently that $P$ (as well as $\boxminus P$) and $E_9$ intersect transversely, is equivalent to proving that the singularities of $B_{S,m}$ lying along $E_9$ are simple nodes.

\begin{definition}Let $S\cong\bl_{\{P_1\dots,P_9\}}\mathbb{P}^2$ be a general rational elliptic surface. We call $V_{\mathcal{B}_m}^{\mathcal{H}_2}(S)\subset V_{|\mathcal{B}_m|,4m+2}(S)$ the Severi variety of irreducible rational curves in $|\mathcal{B}_m|$ with two simple intersection points with $\mathcal{H}_2$. \end{definition}

The Definition of $V_{\mathcal{B}_m}^{\mathcal{H}_2}(S)$ is motivated by the following Remark.

\begin{remark}\label{severires}
    In the Example \ref{invbis} we showed that an origin cutting bisection for a rational elliptic surface $S\rightarrow\mathbb{P}^1$ is of the form \begin{center} $\mathcal{B}_m\sim 6(m+1)L-2(m+1)E_1-\dots-2(m+1)E_8-2mE_9$.\end{center}
    The arithmetic genus of $\mathcal{B}_m$ is \begin{center}
       $\rho(\mathcal{B}_m)=\frac{1}{2}[(6m+5)(6m+4)-8(2m+2)(2m+1)-2m(2m-1)])=4m+2$. 
    \end{center}
    and by Proposition \ref{mpl} the double points of a curve in $|\mathcal{B}_m|$ belong to $\mathcal{H}_2$ or to $E_9$. 
    Moreover,
    \begin{center}
        $\mathcal{B}_m\cdot E_9=2m$  and  $\mathcal{B}_m\cdot\mathcal{H}_2=6(m+1)$.
    \end{center} Let now $B_{S,m}$ be a rational member of $\mathcal{B}_m$ and let us assume that the singularities of $B_{S,m}$ are simple nodes. We have $\rho(B_{S,m})=4m+2$, whence it has $4m+2$ nodes. Since \begin{center} $B_{S,m}\cdot E_9+B_{S,m}\cdot \mathcal{H}_2=8m+6$, \end{center} two free intersection points $Q_1$ and $Q_2$ between $B_{S,m}$ and $\mathcal{H}_2+E_9$ remain. Since $B_{S,m}$ can just have double points as singularities (see Remark \ref{even}) and $B_{S,m}\cdot E_9$ and $B_{S,m}\cdot\mathcal{H}_2$ are even, we have that $Q_1$ and $Q_2$ belong both to $E_9$ or to $\mathcal{H}_2$. In the proof of Corollary \ref{freeh2} we saw that the rational bisections of our interest are the ones intersecting $E_9$ in double points, so we consider the latter case.
    
\end{remark}

For the discussion above, proving the existence of the Enriques surfaces of base change type is equivalent to showing the nonemptiness of $V_{\mathcal{B}_m}^{\mathcal{H}_2}(S)$.
In order to prove that $V_{\mathcal{B}_m}^{\mathcal{H}_2}(S)$ is nonempty, we transfer the problem to the quotient $S/(-1)\cong\mathbb{F}_2$. Once again, the double points of $B_{S,m}$ all belong to $E_9+\mathcal{H}_2$, that is the ramification locus of the quotient map $q$ (as seen while introducing Proposition \ref{quot}), whence the image of $B_{S,m}$, which we will see in Lemma \ref{bisf2} below is a smooth rational curve, is tangent to the branch locus at the images of the singular points.\par

For this reason, we introduce the so-called \textit{logarithmic Severi varieties}, parametrizing curves with given tangency conditions to a fixed curve. The definition and the main results are given by Dedieu in \cite{De} and they are based on the works of Caporaso and Harris (see for example \cite{CH}). \par

Let us denote by $\underline{N}$ the set of sequences $\alpha=[\alpha_1,\alpha_2,\dots]$ of nonnegative integers with all but finitely many $\alpha_i$ non-zero. In practice we shall omit the infinitely many zeroes at the end. For $\alpha\in\underline{N}$, we let \begin{center}
    $|\alpha|=\alpha_1+\alpha_2+\dots$\\
    $\mathcal{I}\alpha=\alpha_1+2\alpha_2+\dots+n\alpha_n+\dots$
\end{center}

\begin{definition}\label{logsev}
Let $S$ be a smooth projective surface, $T\subset S$ a smooth, irreducible curve
and $L$ a line bundle or a divisor class on $S$ with arithmetic genus $\rho$. Let $\delta$ be an integer satisfying $0\leq\delta\leq\rho$, let $\alpha\in\underline{N}$ such that \begin{center}
    $\mathcal{I}\alpha=L\cdot T$.
\end{center} We denote by $V_{\delta,\alpha}(S,T,L)$ the locus of curves in $L$ such that \begin{itemize}
    \item $C$ is irreducible of geometric genus $\delta$ and algebraically equivalent to $L$,
    \item denoting by $\mu:\tilde{C}\rightarrow S$ the normalization of $C$ composed with the inclusion $C\subset S$, there exists $|\alpha|$ points $Q_{i,j}\in C$, $1\leq j\leq\alpha_i$ such that \begin{center}
        $\mu^*T=\sum\limits_{1\leq j\leq\alpha_i}iQ_{i,j}$.
    \end{center}
\end{itemize}
\end{definition}

\begin{theorem}[Dedieu]\label{dedieu}
    Let $V$ be an irreducible component of $V_{\delta,\alpha}(S,T,L)$, $[C]$ a general member of $V$ and $\mu:\tilde{C}\rightarrow S$ its normalization as in the Definition \ref{logsev}. Let now $Q_{i,j}$, $1\leq j\leq\alpha_i$ points in $\tilde{C}$ such that \begin{center}$\mathcal{I}\alpha=L\cdot T$ \end{center} and set \begin{center}$D=\sum\limits_{1\leq j\leq\alpha_i} (i-1)Q_{i,j}$.\end{center}
    \begin{itemize}
        \item[(i)] If $-K_S\cdot C_i-$deg $\mu_*D_{|C_i|}\geq 1$ for every irreducible component $C_i$ of $C$, then \begin{center}
            $\dim(V)=-(K_S+T)\cdot L+\delta-1+|\alpha|$
        \end{center}
        \item[(ii)] If $-K_S\cdot C_i-$deg $\mu_*D_{|C_i|}\geq 2$ for every irreducible component $C_i$ of $C$, then \begin{itemize}
            \item[(a)] the normalization map $\mu$ is an immersion, except possibly at the points $Q_{i,j}$;
            \item[(b)] the points $Q_{i,j}$ of $\tilde{C}$ are pairwise distinct;
            \item[(c)] for every curve $G\subset S$, if $[C]$ is general with respect to $G$ then $C$ intersects $G$ transversely. 
        \end{itemize}

    \end{itemize}
\end{theorem}

In our new setting, the surface we are focusing on is $\mathbb{F}_2$. We call $f$ the class of the members of the ruling and $e$ the special section such that $e^2=-2$.
We describe the images of the origin cutting bisections under the quotient $q:S\rightarrow\mathbb{F}_2$. 
\begin{lemma}\label{bisf2}
    $\mathcal{B}_m$ is sent to $b_m\sim 2(m+1)f+e$ and the branch locus is $e+h_2$, with $h_2\sim 6f+3e$.
\end{lemma}
\begin{proof}
Recall that the ramification locus consists of the union of $E_9$ and $\mathcal{H}_2\sim 9L-3E_1-\dots-3E_8$.
    To prove the Lemma, it is sufficient to describe the map $q^*:\pic(\mathbb{F}_2)\rightarrow\pic(S)$. The quotient map $q$ sends the cubics of the pencil to the lines of the ruling, then $q^*(f)=F\sim 3L-E_1-\dots-E_9$. More over, $E_9$ is sent to $e$ and it belongs to the branch locus, then $q^*(e)=2E_9$. Finally, $\mathcal{B}_m\sim 6(m+1)L-2(m+1)E_1-\dots-2(m+1)E_8-2mE_9$ is sent to $b_m\sim 2(m+1)f+e$ and the component of the ramification locus $\mathcal{H}_2\sim 9L-3E_1-\dots-3E_8\sim 3F+3E_9$ is sent to $6f+3e$. 
\end{proof}

The linear system $b_m$ consists of sections, thus each irreducible member is a smooth rational curve and we shall omit the genus in the notation of the Severi varieties of curves in it. 
We are interested in the images of curves of $V_{\mathcal{B}_m}^{\mathcal{H}_2}(S)$: the singular points of $B_{S,m}$ along $E_9+\mathcal{H}_2$ become tangency points between $q(B_{S,m})\sim b_m$ and $e+h_2$. For this reason, we set \begin{center}
    $\alpha=[2,4m+2]$.
\end{center} Finally, we are interested in the members of $V_{\alpha}(\mathbb{F}_2,e+h_2,b_m)$ intersecting $h_2$ transversely in two points. We call this logarithmic Severi variety $V_{b_m}^{h_2}(\mathbb{F}_2)$.

\begin{theorem}\label{distinct}
     $V_{\mathcal{B}_m}^{\mathcal{H}_2}(S)$ is nonempty of dimension 1.
\end{theorem}
\begin{proof}
    First of all, we prove that the logarithmic Severi variety $V_{b_m}^{h_2}(\mathbb{F}_2)$ of curves on $\mathbb{F}_2$ is nonempty and that its general member is tangent to $e+h_2$ in pairwise distinct points. Finally, we notice that there is a 1:1 correspondence between $V_{\mathcal{B}_m}^{\mathcal{H}_2}(S)$ and $V_{b_m}^{h_2}(\mathbb{F}_2)$. \par
    For every nonnegative integer $m\in\mathbb{Z}_+$, consider a section $E_m\in\mw(S)$ such that $E_m\cdot E_9=m$. The existence of such a section for every $m\in\mathbb{Z}_+$ is ensured by \cite[Theorem 5.2]{Costa}. Denoting by $\boxminus E_m$ the opposite section of $E_m$ with respect to the zero section $E_9$, we have that $E_m$ intersects $\boxminus E_m$ along $E_9$ or $\mathcal{H}_2$. The quotient $q$ identifies $E_m$ and $\boxminus E_m$ by construction. We have that the intersection number between $e_m:=q(E_m)$ and $e$ is \begin{center} $e_m\cdot e=\frac{1}{2}q^*(e_m)\cdot q^*(e)=\frac{1}{2}(E_m+\boxminus E_m)(2E_9)=2m$. \end{center} The curve $e_m$ belongs to $|b_m|$: indeed, $e_m$ is of the form $xf+ye$ and, since it is a section for the ruling, we have $y=1$; furthermore, as just seen, $e_m\cdot e=2m$, which implies $x=2m+2$. The intersection points between $E_m$ and $\boxminus E_m$ along $E_9+\mathcal{H}_2$ become points at which $e_m$ is tangent to $e+h_2$. This implies the nonemptiness of $V_{b_m}^{h_2}(\mathbb{F}_2)$.
    \par
    We set now $D=Q_1+\dots+Q_{4m+2}$. Every irreducible curve of $b_m$ is a section and then is smooth, so it is isomorphic to its normalization. Let $[c]$ be a general member of $V_{b_m}^{h_2}(\mathbb{F}_2)$: by construction $d:=\deg D_{|c}=4m+2$. The anti-canonical class of $\mathbb{F}_2$ is \begin{center}
        $-K_{\mathbb{F}_2}\sim |4f+2e|$.
    \end{center} We have \begin{center}
        $-K_{\mathbb{F}_2}\cdot c-d=(4f+2e)(2(m+1)f+e)-d=4m+4-(4m+2)=2$.  
    \end{center}Thus, by Theorem \ref{dedieu}(ii-b), we conclude that the points at which $[c]$ is tangent to $e$ and $h_2$ are pairwise distinct and that $[c]$ intersects $h_2$ transversely in two distinct points (distinct also from the ones at which it is tangent).\par
    By construction, the pullback of a member of $V_{b_m}^{h_2}(\mathbb{F}_2)$ belongs to $V_{\mathcal{B}_m}^{\mathcal{H}_2}(S)$. This in particular implies that $V_{\mathcal{B}_m}^{\mathcal{H}_2}(S)$ is nonempty. Furthermore, any member of $V_{\mathcal{B}_m}^{\mathcal{H}_2}(S)$ intersects $E_9$ and $\mathcal{H}_2$ in double points except for two points in which it is transverse to $\mathcal{H}_2$, so that it is sent to a member of $V_{b_m}^{h_2}(\mathbb{F}_2)$.\par
    Finally, by Theorem \ref{dedieu}(i), we have\begin{center}  $\dim(V_{\mathcal{B}_m}^{\mathcal{H}_2}(S))=\dim(V_{b_m}^{h_2}(\mathbb{F}_2))=-(K_S+T)\cdot L+\delta-1+|\alpha|=-(|-4f-2e+6f+4e|)\cdot|(2m+2)f+e|-1+4m+4=-|2f+2e|\cdot|(2m+2)f+e|+4m+3=-4m-2+4m+3=1$. \end{center}
\end{proof}

Notice that the nonemptiness of $V_{\mathcal{B}_m}^{\mathcal{H}_2}(S)$ could be also shown by using Corollary \ref{freeh2}: the image of a curve in the set of rational origin cutting bisections in $|\mathcal{B}_m|$ under the quotient $q:S\rightarrow\mathbb{F}_2$ belongs in fact to the logarithmic Severi variety $V_{b_m}^{h_2}(\mathbb{F}_2)$. As explained in the introduction, we preferred to give a purely geometrical lattice free proof of the existence of the Enriques surfaces of base change type, being independent to the one given by Hulek and Sch\"utt.

\begin{remark}\label{curve} 
    It is not surprising that $V_{\mathcal{B}_m}^{\mathcal{H}_2}(S)$ is actually a curve: for every rational elliptic surface $S$ and for every $m$, we expected that we could construct a one-dimensional family of Enriques surfaces of base change type. These curves parametrize the pairs $(Y_m,\mathcal{E}_{Y_m})$ of a $m$-special Enriques surface $Y_m$ and its genus 1 pencil $\mathcal{E}_{Y_m}$ induced by the one of $S$.
\end{remark}

We proved Theorem \ref{nodal}, of which we give a refined statement.

\begin{theorem}\label{bisecnodal}
    Let $S=\bl_{\{P_1,\dots,P_9\}}\mathbb{P}^2$ be a general rational elliptic surface. Denote by $q:S\rightarrow\mathbb{F}_2$ the quotient map over the second Hirzebruch surface $\mathbb{F}_2$ and let $C$ be a general member of $V_{b_m}^{h_2}(\mathbb{F}_2)$. Let us denote $B_{S,m}\in V_{\mathcal{B}_m}^{\mathcal{H}_2}(S)$ the preimage of $C$ under $q$ and $S_{t_0}$ and $S_{t_\infty}$ the fibers to which $B_{S,m}$ is tangent. Let now $X_m$ and $Y_m$ be the K3 surface and the Enriques surface obtained by the base change construction with $S_{t_0}$ and $S_{t_{\infty}}$ as fixed fibers. Let then $g:X_m\rightarrow S$ and $f:X_m\rightarrow Y_m$ denote the corresponding quotients. Finally, denote by $P\in\mw(X)$ the component of $g^{-1}(B_{S,m})$ identified with $\tilde{E}_9$ by $f$ and $B_{Y,m}$ the corresponding $m$-special curve on $Y$.\\ Then, $B_{Y,m}$ is nodal.
\end{theorem}

We show that the $m$-special curves are not unique: for every section $E$ in the Mordell-Weil group of $S$, there exists a nodal rational curve $B_{Y_E}\in Y_m$ of arithmetic genus $m$, which we will refer to also as $m$-special curve.
\begin{lemma}\label{y_e}
    Let $E\in\mw(S)$ be a section for the rational elliptic surface $S$. Then, $\tilde{E}=g^{-1}(E)\in\mw(X)$ is identified with $\tilde{E}\boxplus P$ by $\tau$ and their image $B_{Y_E}:=f(\tilde{E})=f(\tilde{E}\boxplus P)$ is a nodal rational curve of arithmetic genus $m$.
\end{lemma}
\begin{proof}
Let us denote by $\boxplus\tilde{E}\in\aut(X)$ the automorphism of $X$ given by the translation for $\tilde{E}$. The curve $\tilde{E}_9$ is sent to $\tilde{E}$ and $P$ to $\tilde{E}\boxplus P$ by the translation $\boxplus\tilde{E}$. To complete the proof of the first part of the statement, it is sufficient to show that $\boxplus\tilde{E}$ commutes with $\tau$. Let $Q_t\in X_t$: $\boxplus \tilde{E}(\tau(Q_t))=\boxplus\tilde{E}(Q_{-t}\boxplus P_{-t})=Q_{-t}\boxplus P_{-t}\boxplus\tilde{E}_{-t}=\iota(Q_t+\tilde{E}_t-P_t)=\tau(Q_t+\tilde{E}_t)=\tau(\boxplus\tilde{E}(Q_t))$. Since $\tilde{E}+\tilde{E}\boxplus P$ is obtained by translating $E_9+P$ for $\tilde{E}$ and the transversality of the corresponding intersections is preserved under the automorphism $\boxplus\tilde{E}$, we conclude that $B_{Y_E}$ is nodal.
\end{proof}

\subsection{\rm O\sc ther nodal rational curves}

We are able to find other rational curves on $Y_m$. Recall that the Mordell-Weil group of $X$ is generated by $\tilde{E}_i$, with $i\in\{1,\dots,8\}$, and $P$, with neutral element $\tilde{E}_9$. Then, every section in $\mw(X)$ is of the form $\tilde{E}\boxplus P^{\boxplus k}$, with $E\in\mw(S)$.

\begin{lemma}\label{idsections}
    Every section $\tilde{E}\boxplus P^{\boxplus k}\in\mw(X_m)$ is identified by $\tau$ with $\tilde{E}\boxplus P^{\boxplus(1-k)}\in\mw(X_m)$. In particular, the section $P^{\boxplus k}$ is identified with $P^{\boxplus (1-k)}$.
\end{lemma}
\begin{proof}
    The proof is quite similar to the one of Lemma \ref{y_e}. It remains to show that $P^{\boxplus k}$ is identified with $P^{\boxplus (1-k)}$. 
    Let $X_t$ and $X_{-t}$ be two twin fibers. The involution $\tau$ acts on a point $x_t\in X_t$ in the following way:\begin{center}
        $\tau(x_t)=\iota\circ\boxminus P(x_t)=\iota(x_t\boxminus P_t)=x_{-t}\boxplus P_{-t}$.
    \end{center}A point $P^{\boxplus k}_t\in P^{\boxplus k}$ is sent to $\iota(P^{\boxplus(k-1)}_t)=P^{\boxplus (1-k)}_{-t}\in P^{\boxplus (1-k)}$, while a point $P^{\boxplus (1-k)}_t\in P^{\boxplus (1-k)}$ is sent to $\iota(P^{\boxminus k}_t)=P^{\boxplus k}_{-t}$.
\end{proof}
\begin{definition}
    We call $B_{E,k,m}$ the image $f(\tilde{E}\boxplus P^{\boxplus k})=f(\tilde{E}\boxplus P^{\boxplus(1-k)})\subset Y_m$
\end{definition}
\begin{proposition}\label{arit}
    For every $k\in\mathbb{Z_+}$ and for every $E\in\mw(S)$, the curve $B_{E,k,m}\subset Y_m$ is rational of arithmetic genus \begin{center}
        $\rho_a(B_{E,k,m})=(4k^2-4k+1)m+4k^2-4k$.
    \end{center}
\end{proposition}
\begin{proof}
First of all, since $P^{\boxplus k}$ is rational, then its image is. We have that \begin{center} $\rho_a(B_{E,k,m})=\frac{1}{2}B_{E,k,m}^2+1=\frac{1}{2}(\frac{1}{2}(P^{\boxplus k}+P^{\boxplus (1-k)})^2)=$\\ =$\frac{1}{4}(-2-2+2P^{\boxplus k}\cdot P^{\boxplus (1-k)})=\frac{1}{2}P^{\boxplus k}\cdot P^{\boxplus (1-k)}$. \end{center} Moreover, \begin{center}
    $P^{\boxplus k}\cdot P^{\boxplus (1-k)}=P\cdot\tilde{E}_9+P\cdot\tilde{\mathcal{H}}_{2k-1}$,
\end{center} where $\mathcal{H}_{2k-1}$ is the $(2k-1)$-torsion multisection for the elliptic fibration on $S$ and \\$\tilde{\mathcal{H}}_{2k-1}\subset X_m$ is its preimage under $g$. Indeed, the intersection product \begin{center} $P^{\boxplus k}\cdot P^{\boxplus (1-k)}$ \end{center} is the same as \begin{center} $P^{\boxplus (2k-1)}\cdot\tilde{E}_9$ \end{center} and since $\tilde{E}_9$ is the zero-section, this latter intersection is composed by the $(2k-1)$-torsion points (trivial or not) in every fiber. By Proposition \ref{linearclassmult}, we have \begin{center}
    $\mathcal{H}_{2k-1}\sim 12k(k-1)L-4k(k-1)E_1-\dots-4k(k-1)E_8$.
\end{center} Finally, \begin{center}
    $P\cdot\tilde{E}_9+P\cdot\tilde{\mathcal{H}}_{2k-1}=(8k^2-8k+2)m+8k^2-8k$,
\end{center}from which \begin{center}
    $\rho_a(B_{E,k,m})=(4k^2-4k+1)m+4k^2-4k$.
\end{center}
\end{proof}
In order to prove Theorem \ref{nodalk}, we use once again Theorem \ref{dedieu}. 

\begin{proof}[Proof of Theorem \ref{nodalk}]
    As pointed out in the proof of Proposition \ref{arit}, the singularities of $B_{E,k,m}$ are given by the intersection between $P$ and $\tilde{E}_9+\tilde{\mathcal{H}}_{2k-1}$. Thus, proving that $B_{E,k,m}$ is nodal is the same as showing that this latter intersection in transverse, or, equivalently, that the intersection between $B_{S,m}\subset S$ and $E_9+\mathcal{H}_{2k-1}$ is transverse. Consider once again the quotient $q:S\rightarrow\mathbb{F}_2$ and denote by $h_{2k-1}$ the image of the torsion multisection $\mathcal{H}_{2k-1}$. We already proved in Theorem \ref{distinct} that for a general choice of $[c]\in V_{b_m}^{h_2}(\mathbb{F}_2)$, the bisection $B_{S,m}$ intersects $E_9$ transversely. By Theorem \ref{dedieu}(ii-c), the general member of $V_{b_m}^{h_2}(\mathbb{F}_2)$ intersects $h_{2k-1}$ transversely. By varying $k$ in $\mathbb{Z}_+$, we have the assertion.
\end{proof}
The reason why we need to relax the hypothesis of generality by requiring the very generality of $Y_m$ is the following. Since for every $k$ we have that the general member of $V_{b_m}^{h_2}(\mathbb{F}_2)$ intersects $h_{2k-1}$ transversely, then for every $k$ we are removing a finite number of choices of points of $V_{b_m}^{h_2}(\mathbb{F}_2)$. But $k$ varies in $\mathbb{Z}_+$ that is countable, so a priori we are removing a countable set of points from $V_{b_m}^{h_2}(\mathbb{F}_2)$. 
\begin{remark}
 Hulek and Sch\"utt in \cite[Lemma 4.1]{HS} prove that there is an injection $\mw(S)\hookrightarrow\aut(Y_m)$ between the group of sections of $S$ and the automorphism group of $Y_m$. In particular, the injection is induced by the map \begin{center}
    $E\mapsto\boxplus\tilde{E}\in\aut(X_m)$,
\end{center} that is well-defined since the translation $\boxplus\tilde{E}$ in $X_m$ commutes with the Enriques involution $\tau$. Denoting by $\epsilon:Y_m\rightarrow\mathbb{P}^1$ the genus 1 pencil of $Y_m$ induced by the one of the rational elliptic surface $S$, Proposition \ref{arit} and Theorem \ref{nodalk} imply that for every $k\in\mathbb{Z}_+$ the set of the nodal rational bisections $B_{E,k,m}$ for $\epsilon$ injects in $\aut(Y_m)$. In other words, the bisections $B_{E,k,m}$ form a rank 8 subgroup of $\aut(Y_m)$. We can conclude that, for every $k\in\mathbb{Z}_+$, the nodal rational bisection $B_{E,k,m}$ is unique up to automorphisms: given two bisections $B_1:=B_{E_1,k,m}$ and $B_2:=B_{E_2,k,m}$ for some $E_1$ and $E_2$ in $\mw(S)$, the automorphism of $Y_m$ sending the former to the latter is the one induced by the (translation for the) section $\tilde{E}_2\boxminus \tilde{E}_1\in\mw(X_m)$.
\end{remark}

  \end{document}